\documentclass[letterpaper,11pt,twoside,keywordsasfootnote,addressatend,noinfoline]{article}
\usepackage{fullpage}
\usepackage[english]{babel}
\usepackage{amssymb}
\usepackage{amsmath}
\usepackage{theorem}
\usepackage{epsfig}
\usepackage{subfigure}
\usepackage{mathrsfs}
\usepackage{color}
\usepackage{imsart}

\theorembodyfont{\slshape}

\newtheorem{theorem}{Theorem}

\newtheorem{lemma}[theorem]{Lemma}

\newenvironment{proof}{\noindent{\bf Proof.}}{\hspace*{2mm}~$\square$}

\newcommand{\R}{\mathbb{R}}

\newcommand{\A}{\mathscr{A}}
\newcommand{\C}{\mathscr{C}}
\newcommand{\E}{\mathscr{E}}
\newcommand{\F}{\mathscr{F}}
\newcommand{\G}{\mathscr{G}}
\newcommand{\V}{\mathscr{V}}
\newcommand{\n}{\hspace*{-6pt}}
\newcommand{\ep}{\epsilon}

\newcommand{\norm}[1]{|\!|#1|\!|}
\newcommand{\bignorm}[1]{\bigg|\!\bigg|#1 \,\bigg|\!\bigg|}

\renewcommand{\r}{\mathbf{r}}
\renewcommand{\c}{\mathbf{c}}

\DeclareMathOperator{\card}{card}
\DeclareMathOperator{\uniform}{Uniform}
\DeclareMathOperator{\conv}{Conv}
\DeclareMathOperator{\Int}{Int}

\begin{document}

\begin{frontmatter}
\title     {Consensus in the Hegselmann-Krause model}
\runtitle  {Consensus in the Hegselmann-Krause model}
\author    {Nicolas Lanchier\thanks{Nicolas Lanchier was partially supported by NSF grant CNS-2000792.} and Hsin-Lun Li}
\runauthor {Nicolas Lanchier and Hsin-Lun Li}
\address   {School of Mathematical and Statistical Sciences \\ Arizona State University, Tempe, AZ 85287, USA. \\ nicolas.lanchier@asu.edu \\ hsinlunl@asu.edu}

\maketitle

\begin{abstract} \ \
 This paper is concerned with the probability of consensus in a multivariate, spatially explicit version of the Hegselmann-Krause model for the dynamics of opinions.
 Individuals are located on the vertices of a finite connected graph representing a social network, and are characterized by their opinion, with the set of opinions~$\Delta$ being a general bounded convex subset of a finite dimensional normed vector space.
 Having a confidence threshold~$\tau$, two individuals are said to be compatible if the distance~(induced by the norm) between their opinions does not exceed the threshold~$\tau$.
 Each vertex~$x$ updates its opinion at rate the number of its compatible neighbors on the social network, which results in the opinion at~$x$ to be replaced by a convex combination of the opinion at~$x$ and the nearby opinions:~$\alpha$ times the opinion at~$x$ plus~$(1 - \alpha)$ times the average opinion of its compatible neighbors.
 The main objective is to derive a lower bound for the probability of consensus when the opinions are initially independent and identically distributed with values in the opinion set~$\Delta$.
\end{abstract}

\begin{keyword}[class=AMS]
\kwd[Primary ]{60K35}
\end{keyword}

\begin{keyword}
\kwd{Interacting particle systems, Hegselmann-Krause model, opinion dynamics, martingale, martingale convergence theorem, optional stopping theorem, confidence threshold, consensus.}
\end{keyword}

\end{frontmatter}


\section{Introduction}
 As illustrated by the popular Axelrod model~\cite{axelrod_1997}, opinion and cultural dynamics are driven by two key components:
 social influence, the tendency of individuals to become more similar when they interact, and homophily, the tendency to interact more frequently with individuals who are more similar~\cite{castellano_fortunato_loreto_2009}.
 Another important component in the context of social dynamics is the topology of the social network whose connections indicate who interacts with whom, an aspect that is well captured by stochastic models known as interacting particle systems~\cite{liggett_1985, liggett_1999}. \vspace*{5pt} \\
{\bf The voter model}.
 The simplest example of interacting particle system modeling the dynamics of opinions is the voter model introduced independently in~\cite{clifford_sudbury_1973, holley_liggett_1975}.
 Individuals located on the vertices of a connected graph~(interpreted as a social network) are characterized by one of two competing opinions, and update their opinion independently at rate one by mimicking one of their neighbors chosen uniformly at random.
 In particular, this model accounts for the topology of the social network as well as social influence, but not homophily.
 Most of the results about the voter model rely on a duality relationship with a system of coalescing random walks.
 When the individuals are located on the~$d$-dimensional integer lattice, the system clusters, and so a local consensus is reached, in dimensions~$d \leq 2$, whereas the process converges to a nontrivial stationary distribution, and so the two opinions coexist at equilibrium, in higher dimensions~\cite{holley_liggett_1975}.
 The asymptotic behavior of the cluster size for the one- and the two-dimensional voter models is studied in~\cite{bramson_griffeath_1980, cox_griffeath_1986}.
 In higher dimensions, even though the two opinions coexist at equilibrium, as time evolves, spatial correlations build up due to the presence of local interactions.
 The spatial correlations in the infinite time limit are studied in~\cite{bramson_griffeath_1979, zahle_2001}.
 On finite connected graphs, the voter model always reaches a consensus, and the time to consensus on the large~$d$-dimensional torus is studied in~\cite{cox_1989}. \vspace*{5pt} \\
\newpage \noindent
{\bf Threshold voter models}.
 Natural variants of the voter model that also describe the dynamics of opinions on social networks are the threshold voter models.
 In this case, individuals switch their opinion at rate one if and only if they disagree with at least~$\theta$ of their neighbors.
 Looking again at the process on the~$d$-dimensional integer lattice, when the threshold is equal to one, coexistence occurs except in one dimension~\cite{liggett_1994}.
 While increasing the threshold~$\theta$, the process exhibits two phase transitions, from coexistence to clustering, and then from clustering to fixation where both opinions again coexist, but coexistence is now due to the system becoming frozen locally as opposed to convergence to a unique nontrivial stationary distribution.
 Fixation occurs in particular when the individuals switch their opinion if and only if they disagree with a strict majority of their neighbors~\cite{durrett_steif_1993}, a particular case referred to as the majority vote model. \vspace*{5pt} \\
{\bf The majority rule model}.
 Another process related to the majority vote model that includes the simultaneous updates of multiple individuals is the majority rule model~\cite{galam_2002}.
 In the original version of this model, the population is finite, individuals are again characterized by one of two competing opinions, and a single update consists in choosing a random number of individuals uniformly~(referred to as a discussion group), which results in all the individuals in the group adopting the majority opinion within the group before the interaction.
 In case of a tie, all the individuals adopt a preferred opinion representing the status quo.
 A spatial version of the majority rule model where individuals are located on the integer lattice is introduced in~\cite{lanchier_neufer_2013}.
 The discussion groups consist of all the~$n \times \cdots \times n$ hypercubes and are chosen to be updated independently at rate one.
 In this case, the behavior depends on the parity of~$n$: the system clusters with the density of each opinion being preserved when~$n$ is odd whereas the status quo wins when~$n$ is even. \vspace*{5pt} \\
{\bf The Axelrod model}.
 The classical and threshold voter models and the majority rule model account for social influence but not homophily, the tendency to interact more frequently with individuals who are more similar.
 The most popular interacting particle system that includes both social influence and homophily is the Axelrod model~\cite{axelrod_1997}.
 Individuals are now characterized by one of~$q$ possible opinions about~$F$ different issues, which results in~$q^F$ possible vectors of opinions representing cultural profiles.
 Homophily is modeled by assuming that neighbors interact at a rate equal to the fraction of opinions or cultural features they have in common while social influence is modeled by matching one of the opinions the two neighbors do not share, if any.
 Depending on the values of~$F$ and~$q$, the system either clusters so that a local consensus is reached or fixates in a fragmented configuration where disagreements persist.
 Intuitively, increasing the value of~$F$ promotes consensus while increasing the value of~$q$ promotes discordance.
 For the one-dimensional model where the individuals are located on the integers, consensus is studied in~\cite{lanchier_2012a, lanchier_schweinsberg_2012} while discordance is studied in~\cite{lanchier_scarlatos_2013}.
 See also~\cite{li_2014} for a study of the two-dimensional process. \vspace*{5pt} \\
{\bf Constrained voter models}.
 Homophily in the Axelrod model is incorporated by assuming that the rate at which individuals interact is proportional to the amount of agreement between neighbors.
 In many other models, this component is modeled instead through the inclusion of a so-called confidence threshold: neighbors either interact at a constant rate or do not interact at all depending on whether their level of disagreement is within a reasonable threshold or not.
 The simplest such model is the constrained voter model~\cite{vazquez_krapivsky_redner_2003} where individuals are either leftist, centrist or rightist, and mimic a randomly chosen neighbor unless one of the individuals is leftist and the other one rightist.
 This model has been generalized in~\cite{lanchier_scarlatos_2017} where the set of opinions is represented by the vertices of a general finite connected graph and where the level of disagreement is measured using the geodesic distance on the opinion graph:
 neighbors interact at rate one, which results in one neighbor mimicking the other neighbor if and only if the geodesic distance~(or graph distance) between their opinions does not exceed some confidence threshold~$\tau$.
 It is proved in~\cite{lanchier_scarlatos_2017} that the one-dimensional process clusters when the confidence threshold exceeds the radius of the opinion graph while sufficient conditions for fixation in a fragmented configuration are also given.
 A lower bound for the probability of consensus for the process where individuals are located on a finite connected graph is derived in~\cite{hardin_lanchier_2021}.
 The constrained voter model~\cite{vazquez_krapivsky_redner_2003} corresponds to the particular case where the opinion graph consists of the path with three vertices and two edges and where~$\tau = 1$ while the particular case where the opinion graph consists of the hypercube with~$2^n$ vertices is referred to as the vectorial Deffuant model introduced in~\cite{deffuant_2000} and studied in~\cite{lanchier_scarlatos_2014}. \vspace*{5pt} \\
{\bf The Deffuant model}.
 In another~(more popular) version of the Deffuant model~\cite{deffuant_2000}, the opinion space is uncountable and consists of the unit interval.
 Edges become active independently at rate one and the two individuals connected by that edge interact if and only if the distance between their opinions does not exceed some confidence threshold~$\tau$, which results in the two neighbors' opinions getting closer to each other.
 If the opinions are initially independent and uniform, it is conjectured that the system on infinite graphs converges to a consensus when~$\tau > 1/2$ whereas disagreements persist when~$\tau < 1/2$.
 This conjecture was first established in~\cite{lanchier_2012b} for the one-dimensional process while~\cite{haggstrom_2012} gave an alternative simpler proof.
 Multivariate versions of the model where the opinion space consists of a bounded convex subset of a normed vector space are also studied in~\cite{hirscher_2014} for the one-dimensional system and in~\cite{lanchier_li_2020} for the process on general finite connected graphs. \vspace*{5pt} \\
{\bf The Hegselmann-Krause model}.
 The Hegselmann-Krause (HK) model~\cite{hegselmann_krause_2002} is a fairly general interacting particle system.
 The opinion space again consists of the unit interval but, in contrast with the Deffuant model which evolves via pairwise interactions, the new opinion at a vertex~$x$ is now replaced by a weighted average of all the opinions in the population that are within some confidence threshold~$\tau$ of the opinion at~$x$.
 General conditions under which certain initial configurations and certain sets of parameters can lead to asymptotic stability and/or consensus are discussed in~\cite{li_2020, li_2021}.
 In the version of the HK-model considered in~\cite{castellano_fortunato_loreto_2009}, the matrix of the weights between individuals is chosen to be the adjacency matrix of a graph~(the weight is one if the two individuals are connected by an edge and zero otherwise) thus turning the model into a spatial model:
 Individuals are located on the vertices of a graph representing a social network, and each update at vertex~$x$ results in the new opinion at~$x$ to be the average of the opinions of the neighbors that are compatible with~$x$~(at opinion distance less than~$\tau$).
 The main objective of this paper is to derive a lower bound for the probability of consensus for a multivariate version of this model.


\section{Model description and results}
 To define the model rigorously, let~$\G = (\V, \E)$ be a finite connected graph representing a social network: each vertex is occupied by an individual and any two individuals are connected by an edge
 if and only if they are friends.
 Individuals are characterized by their opinion, and we assume that the opinions are initially independent and distributed across a bounded convex subset~$\Delta \subset \R^n$, thus extending the opinion space~$[0, 1]$ of the original HK-model.
 To measure the disagreements between neighbors and define the dynamics of the process, we also assume that~$\R^n$ is equipped with an arbitrary norm~$\norm{\cdot}$.
 The state of the process at time~$t$ is a spatial configuration
 $$ \xi_t : \V \to \Delta \quad \hbox{where} \quad \xi_t (x) = \hbox{opinion at vertex~$x$ at time~$t$}. $$
 Having a confidence threshold~$\tau > 0$, two neighbors are said to be compatible at time~$t$ if the distance~(induced by the norm) between their opinions does not exceed this threshold, and we denote by~$N_t (x)$ the set of neighbors of~$x$ that are compatible with~$x$, i.e.,
 $$ N_t (x) = \{y \in \V : (x, y) \in \E \ \hbox{and} \ \norm{\xi_t (x) - \xi_t (x)} \leq \tau \}. $$
 The process we study evolves as follows:
 Individuals update their opinion independently at rate the number of their compatible neighbors, the new opinion being equal to the average opinion of their compatible neighbors.
 More precisely, letting~$\F_t$ be the natural filtration of the process,
\begin{equation}
\label{eq:HK1}
  \lim_{\ep \to 0} \ \frac{1}{\ep} \ P \bigg(\xi_{t + \ep} (x) =  \frac{1}{|N_t (x)|} \sum_{y \in N_t (x)} \xi_t (y) \ \Big| \ \F_t \bigg) = |N_t (x)|
\end{equation}
 when~$N_t (x) \neq \varnothing$ whereas the transition rate is zero when~$x$ has no compatible neighbors.
 These transitions indicate that the individuals are open-minded in the sense that the new opinion only depends on the opinions of the compatible neighbors.
 Following~\cite{li_2020, li_2021}, we can include more realistically a degree of stubbornness by assuming that the new opinion is given by
 $$ \alpha \,\xi_t (x) + \frac{1 - \alpha}{|N_t (x)|} \sum_{y \in N_t (x)} \xi_t (y) \quad \hbox{where} \quad \alpha \in [0, 1]. $$
 The parameter~$\alpha$ represents the degree of stubbornness of the individuals.
 To simplify the algebra, we assume from now on that~$\alpha = 0$, which corresponds to model~\eqref{eq:HK1}, but we point out that all our proofs and our main result easily extend to the model with stubbornness.
 The main objective is to study the probability of consensus, i.e., the probability that all the opinions converge to the same limit, when starting from a configuration where the opinions are independent and identically distributed with values in the opinion space~$\Delta$.
 To state our result, let
 $$ \begin{array}{l} \C = \Big\{\lim_{t \to \infty} \max_{x, y \in \V} \norm{\xi_t (x) - \xi_t (y)} = 0 \Big\} \end{array} $$
 be the consensus event.
 We also define
 $$ \begin{array}{rcl}
     \r & \n = \n & \inf \{r > 0 : \Delta \subset B (c, r) \ \hbox{for some} \ c \in \Delta \} \vspace*{4pt} \\
        & \n = \n & \inf \{r > 0 : \sup_{a \in \Delta} \norm{a - c} \leq r \ \hbox{for some} \ c \in \Delta \}, \end{array} $$
 which can be viewed as the radius of the opinion space, and
 $$ \c \in \Delta \quad \hbox{such that} \quad \Delta \subset B (\c, \r), $$
 which can be viewed as the center of the opinion space.
 Throughout this paper, we also assume that the opinions across the social network are intially independent and identically distributed, and we let~$X$ be the~$\Delta$-valued random variable describing the initial opinions:
 $$ P (\xi_0 (x) \in B) = P (X \in B) \quad \hbox{for all} \quad x \in \V \ \hbox{and} \ \hbox{Borel set} \ B \subset \Delta. $$
 Then, we have the following universal lower bound that holds uniformly over all possible finite connected graphs~(social networks).
\begin{theorem}[probability of consensus] --
\label{th:consensus}
 For all~$\tau > \r$, $P (\C) \geq 1 - E \,\norm{X - \c} / (\tau - \r)$.
\end{theorem}
 This result is identical to~\cite[Theorem~2.1]{lanchier_li_2020} which applies to the Deffuant model.
 However, because the HK dynamics is significantly different from the Deffuant dynamics, our proof is quite different from the proof in~\cite{lanchier_li_2020}.
 To begin with, we prove that, for all~$c \in \R^n$, the process~$X_t (c)$ that keeps track of the total disagreement between a virtual external agent with opinion~$c$ and the vertices of the graph is a supermartingale~(Lemma~\ref{lem:supermartingale}).
 Looking at~$2^d$ of these processes for values of~$c$ that surround the opinion space and applying the martingale convergence theorem, we deduce that the opinion model converges almost surely to a random configuration~$\xi_{\infty}$~(Lemma~\ref{lem:limit}).
 The evolution rules of the HK-model imply that, in the configuration~$\xi_{\infty}$, any two neighbors on the social network either completely agree or disagree by more than the confidence threshold~$\tau$~(Lemma~\ref{lem:constant}).
 Using this characterization of~$\xi_{\infty}$, we deduce that, for all~$\ep > 0$, the first time~$T_{\ep}$ at which any two neighbors disagree by at most~$\ep$ or at least~$\tau$ is almost surely finite~(Lemma~\ref{lem:stopping-time}).
 We also prove that the event~$\A$ that the opinion at a vertex is at distance less than~$\tau - \r$ from the center~$\c$ of the opinion space at time~$T_{\ep}$ for some~$\ep$ small always drives the process to a global consensus~(Lemma~\ref{lem:consensus}).
 Applying the optional stopping theorem to the supermartingale~$X_t (\c)$ stopped at time~$T_{\ep}$ gives a lower bound for the probability of~$\A$~(Lemma~\ref{lem:subevent}) which, together with Lemma~\ref{lem:consensus}, implies the theorem. \vspace*{5pt} \\
{\bf Numerical example}.
 Before going into the details of the proof, we give a numerical example to show how our result applies to the original HK-model.
 Assume that the opinion space is
 $$ \Delta = B (a, r) = \{c \in \R^n : \norm{a - c} < r \}, $$
 the ball with center~$a$ and radius~$r$, and that the opinions are initially independent and uniformly distributed: $X = \uniform (\Delta)$.
 Then,~$\c = a$ and~$\r = r$.
 In addition,
 $$ P (\norm{X - a} \leq s) = P (X \in B (a, s)) = \frac{\lambda (B (0, s))}{\lambda (B (0, r))} = \frac{s^n \,\lambda (B (0, 1))}{r^n \,\lambda (B (0, 1))} = \bigg(\frac{s}{r} \bigg)^n $$
 for all~$s \leq r$, from which it follows that
 $$ E \,\norm{X - a} = \int_0^{\infty} P (\norm{X - a} > s) \,ds = \int_0^r \bigg(1 - \bigg(\frac{s}{r} \bigg)^n \bigg) \,ds = \frac{nr}{n + 1}. $$
 Therefore, the theorem implies that
 $$ P (\C) \geq 1 - \frac{E \,\norm{X - \c}}{\tau - \r} = 1 - \frac{E \,\norm{X - a}}{\tau - r} = 1 - \frac{nr}{(n + 1)(\tau - r)} \quad \hbox{when} \quad \Delta = B (a, r). $$
 In particular, for the original HK-model~($n = 1$ and~$a = r = 1/2$),
 $$ P (\C) \geq \frac{4 \tau - 3}{4 \tau - 2} > 0 \quad \hbox{for all} \quad \tau > 3/4 \quad \hbox{when} \quad \Delta = [0, 1]. $$


\section{Underlying supermartingales}
 In this section, we prove that the processes that keep track of the total disagreement between a fixed point~$c$ and the model's opinions, i.e.,
 $$ X_t (c) = \sum_{x \in \V} \,\norm{\xi_t (x) - c \,} \quad \hbox{for all} \quad c \in \R^n, $$
 are supermartingales.
 Applying the martingale convergence theorem to these supermartingales will show almost sure convergence of the HK-model to a limiting configuration~$\xi_{\infty}$.
 This result will also be used later in combination with the optional stopping theorem to derive the lower bound for the probability of consensus in the theorem.
\begin{lemma} --
\label{lem:supermartingale}
 The processes~$X_t (c)$ are supermartingales for the filtration~$\F_t$.
\end{lemma}
\begin{proof}
 Let~$\V_t$ be the set of vertices with at least one compatible neighbor, i.e., at least one neighbor at opinion distance at most~$\tau$ of the vertex.
 To simplify the notation, we let
 $$ \bar \xi_t (x) = \frac{1}{|N_t (x)|} \sum_{y \in N_t (x)} \xi_t (y) \quad \hbox{for all} \quad x \in \V_t. $$
 By the absolute homogeneity of the norm and the triangle inequality, for all~$x \in \V_t$,
 $$ \begin{array}{l}
    \displaystyle |N_t (x)| \cdot \norm{\bar \xi_t (x) - c \,} =
    \displaystyle |N_t (x)| \cdot \bignorm{\frac{1}{|N_t (x)|} \sum_{y \in N_t (x)} \xi_t (y) - c \,} \vspace*{4pt} \\ \hspace*{40pt} =
    \displaystyle \bignorm{\sum_{y \in N_t (x)} \xi_t (y) - |N_t (x)| \cdot c \,} =
    \displaystyle \bignorm{\sum_{y \in N_t (x)} (\xi_t (y) - c) \,} \leq \sum_{y \in N_t (x)} \norm{\xi_t (y) - c \,}. \end{array} $$
 In particular, recalling the transition rates of the HK-model, we get
 $$ \begin{array}{rcl}
    \displaystyle \lim_{\ep \to 0} \,E \bigg(\frac{X_{t + \ep} (c) - X_t (c)}{\ep} \ \Big| \ \F_t \bigg) & \n = \n &
    \displaystyle \sum_{x \in \V_t} \,|N_t (x)| \cdot \Big(\norm{\bar \xi_t (x) - c \,} - \norm{\xi_t (x) - c \,} \Big) \vspace*{4pt} \\ & \n = \n &
    \displaystyle \sum_{x \in \V_t} \left(\sum_{y \in N_t (x)} \norm{\xi_t (y) - c \,} - |N_t (x)| \cdot \norm{\xi_t (x) - c \,} \right). \end{array} $$
 To conclude, observe that
 $$ \begin{array}{rcl}
     |\{x \in \V_t : y \in N_t (x) \}| & \n = \n &
     |\{x : (x, y) \in \E \ \hbox{and} \ \norm{\xi_t (x) - \xi_t (y)} \leq \tau \}| \vspace*{4pt} \\ & \n = \n &
     |\{x : (y, x) \in \E \ \hbox{and} \ \norm{\xi_t (y) - \xi_t (x)} \leq \tau \}| = |N_t (y)|. \end{array} $$
 In words, the number of vertices~$y$ is in the neighborhood of and compatible with is equal to the number of vertices that are in the neighborhood of~$y$ and compatible with~$y$.
 It follows that
 $$ \begin{array}{l}
    \displaystyle \lim_{\ep \to 0} \,E \bigg(\frac{X_{t + \ep} (c) - X_t (c)}{\ep} \ \Big| \ \F_t \bigg) \leq
    \displaystyle \sum_{x \in \V_t} \ \sum_{y \in N_t (x)} \norm{\xi_t (y) - c \,} - \sum_{x \in \V_t} \,|N_t (x)| \cdot \norm{\xi_t (x) - c \,} \vspace*{8pt} \\ \hspace*{50pt} =
    \displaystyle \sum_{y \in \V_t} |\bar \xi_t (y)| \cdot \norm{\xi_t (y) - c \,} - \sum_{x \in \V_t} \,|N_t (x)| \cdot \norm{\xi_t (x) - c \,} = 0, \end{array} $$
 which proves that~$X_t (c)$ is a supermartingale.
\end{proof} \\ \\
 Using Lemma~\ref{lem:supermartingale}, we now prove almost sure convergence of the HK-model.
\begin{lemma} --
\label{lem:limit}
 There is a random configuration~$\xi_{\infty} : \V \to \Delta$ such that
 $$ \lim_{t \to \infty} \xi_t (x) = \xi_{\infty} (x) \quad \hbox{almost surely} \quad \hbox{for all} \ x \in \V. $$
\end{lemma}
\begin{proof}
 Let~$r > 0$ be sufficiently large that~$\Delta \subset \Lambda = [- r, r]^d$ (see Figure~\ref{fig:opinion}).
 Assuming that the configuration~$\xi_t$ does not converge, there exists~$\ep > 0$ such that
\begin{equation}
\label{eq:limit-1}
  \max_{x \in \V} \norm{\xi_t (x) - \xi_{t-} (x)}_1 = \max_{x \in \V} \ \sum_{i = 1}^n \,|\langle \xi_t (x) - \xi_{t-} (x), e_i \rangle| > \ep \quad \hbox{at arbitrary large times}
\end{equation}
 where~$e_i = (0, \ldots, 0, 1, 0, \ldots, 0) \in \R^n$ denotes the~$i$th unit vector.
 By the choice of~$r$ and by the equivalence of the norms in finite dimension, there exists~$c_1 > 0$ such that, for at least one of the corners of the~$n$-dimensional cube~$\Lambda$, say corner~$c$,
 $$ \begin{array}{rcl}
     |X_t (c) - X_{t-} (c)| & \n \geq \n &
      c_1 \,\max_{x \in \V} |\norm{\xi_t (x) - c \,}_1 - \norm{\xi_{t-} (x) - c \,}_1| \vspace*{4pt} \\ & \n = \n &
      c_1 \,\max_{x \in \V} \norm{\xi_t (x) - \xi_{t-} (x)}_1 > \ep c_1 > 0. \end{array} $$
 In particular, for at least one corner~$c$,
\begin{equation}
\label{eq:limit-2}
  |X_t (c) - X_{t-} (c)| > \ep c_1 \quad \hbox{at arbitrary large times},
\end{equation}
 showing that~$X_t (c)$ does not converge.
 Now, using again the equivalence of the norms, there exists~$c_2 < \infty$ such that the process is bounded by
 $$ |X_t (c)| \leq \max_{a \in \Delta} \norm{a - c \,} \card (\V) \leq c_2 \,\max_{a, b \in \Lambda} \norm{a - b}_1 \card (\V) = 2nrc_2 \card (\V) < \infty. $$
 Because the process~$X_t (c)$ is also a supermartingale according to Lemma~\ref{lem:supermartingale}, it follows from the martingale convergence theorem that it converges almost surely.
 In particular, the events in~\eqref{eq:limit-1} and~\eqref{eq:limit-2} and the divergence of~$\xi_t$ occur with probability zero, which proves the lemma.
\end{proof} \\ \\
 To study the limiting configuration~$\xi_{\infty}$, we now let~$\G_{\infty} = (\V, \E_{\infty})$ be the random subgraph of the social network~$\G$ induced by the edges that are compatible in the limit:
\begin{equation}
\label{eq:subedge}
  (x, y) \in \E_{\infty} \quad \hbox{if and only if} \quad (x, y) \in \E \quad \hbox{and} \quad \norm{\xi_{\infty} (x) - \xi_{\infty} (y)} \leq \tau.
\end{equation}
 The next lemma shows that, in the limiting configuration, individuals on the same connected component of the graph~$\G_{\infty}$ share the same opinion.
\begin{figure}[t!]
\centering
\scalebox{0.50}{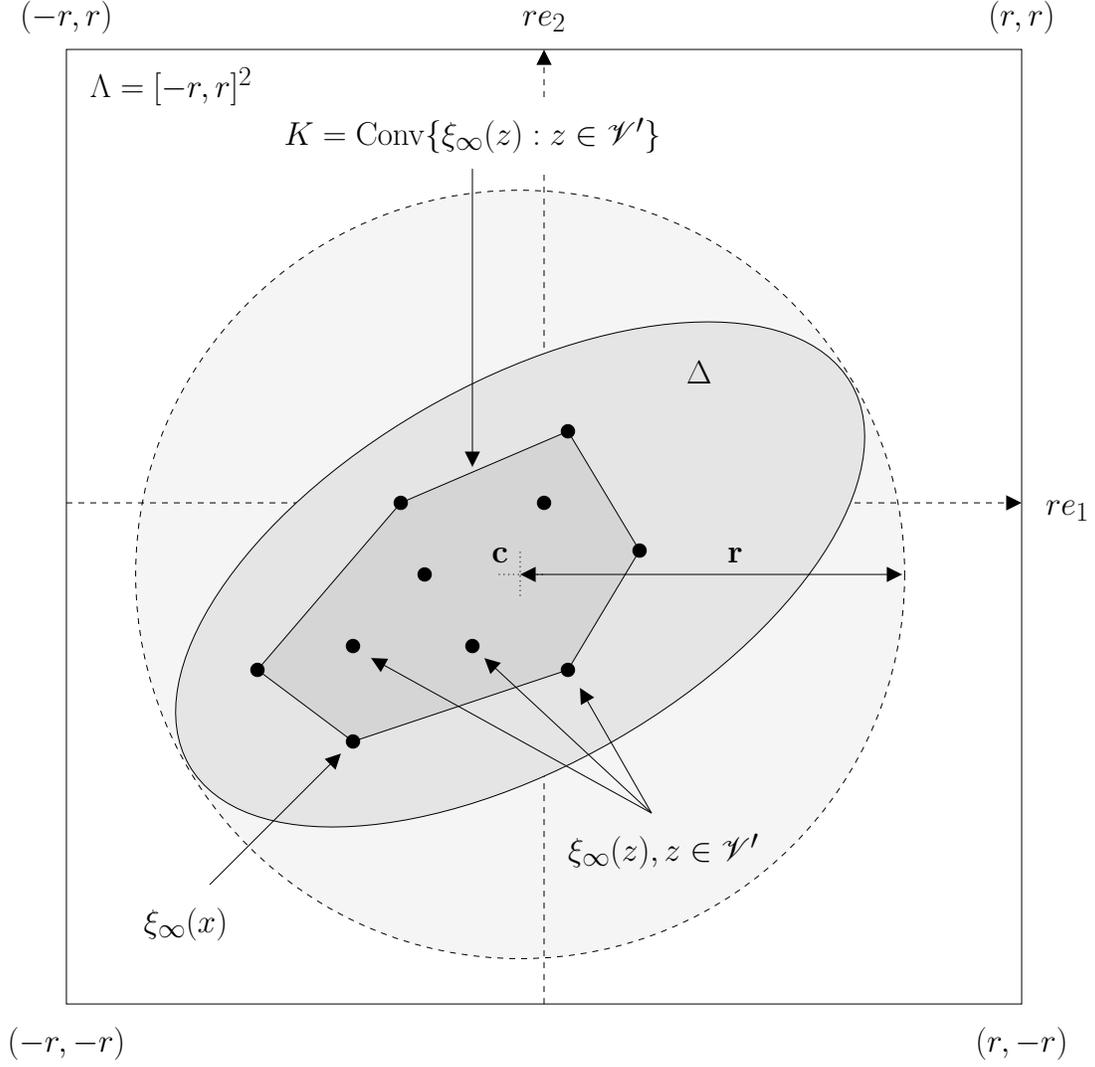}
\caption{\upshape{Picture related to the proof of Lemmas \ref{lem:limit} and~\ref{lem:constant}}.}
\label{fig:opinion}
\end{figure}
\begin{lemma} --
\label{lem:constant}
 Configuration~$\xi_{\infty}$ is constant on each of the connected components of~$\G_{\infty}$.
\end{lemma}
\begin{proof}
 Let~$\G' = (\V', \E')$ be a connected component of the graph~$\G_{\infty}$, and observe that, because~$\xi_{\infty}$ is an absorbing state for the HK-model, it follows from the evolution rules that
 $$ N_{\infty} (x) = \lim_{t \to \infty} N_t (x) = \varnothing \quad \hbox{or} \quad \xi_{\infty} (x) = \frac{1}{\card (N_{\infty} (x))} \sum_{y \in N_{\infty} (x)} \xi_{\infty} (y) $$
 for all~$x \in \V$.
 When~$\V'$ reduces to one vertex,~$\xi_{\infty}$ is trivially constant on~$\V'$.
 Otherwise, because the subgraph~$\G'$ is connected and all its edges are compatible in the limit, we have
\begin{equation}
\label{eq:constant-1}
  N_{\infty} (x) \neq \varnothing \quad \hbox{so} \quad \xi_{\infty} (x) = \frac{1}{\card (N_{\infty} (x))} \sum_{y \in N_{\infty} (x)} \xi_{\infty} (y) \quad \hbox{for all} \quad x \in \V'.
\end{equation}
 In this case, let~$K$ be the convex envelope of the limiting opinions in~$\V'$,
 $$ K = \conv \{\xi_{\infty} (z) : z \in \V' \} \subset \Delta, $$
 and let~$\xi_{\infty} (x) \in K$, $x \in \V'$, be an extreme point of~$K$, i.e., a point which is not an inner point of segments of the convex envelope~(see Figure~\ref{fig:opinion}).
 Now, assume by contradiction that
\begin{equation}
\label{eq:constant-2}
  \xi_{\infty} (y) \neq \xi_{\infty} (x) \quad \hbox{for some} \quad y \in N_{\infty} (x).
\end{equation}
 Then, it follows from~\eqref{eq:constant-1} that
\begin{equation}
\label{eq:constant-3}
  \xi_{\infty} (x) \in \Int (K') \quad \hbox{where} \quad K' = \conv \{\xi_{\infty} (y) : y \in N_{\infty} (x) \}.
\end{equation}
 Here, $\Int (K')$ refers to the interior of the set~$K'$. Because
 $$ N_{\infty} (x) \subset \V' \quad \hbox{and so} \quad K' \subset K, $$
 we deduce from~\eqref{eq:constant-3} that~$\xi_{\infty} (x) \in \Int (K)$, contradicting the fact that~$\xi_{\infty} (x)$ is an extreme point of the convex set~$K$.
 This shows that the statement in~\eqref{eq:constant-2} is false therefore all the neighbors of~$x$ on the connected component must share the same opinion as~$x$ in the limit.
 This also implies that the neighbors' opinions are (identical) extreme points of~$K$ so, by the same reasoning, the neighbors of the neighbors of~$x$ must share the same opinion as~$x$, and using a simple induction, we deduce that the convex envelope reduces to one point: $K = \{\xi_{\infty} (x) \}$, which proves the lemma. 
\end{proof} \\ \\
 Note that the proof of the lemma can be simplified when dealing with the original HK-model, where~$\Lambda = (0, 1)$ and~$\norm{\cdot}$ refers to the Euclidean norm, using that, according to~\eqref{eq:constant-1}, the restriction of the limiting configuration to each of the connected components of~$\G_{\infty}$ is harmonic.
 Indeed, in this case, let~$x \in \V'$ be a vertex where~$\xi_{\infty}$ reaches its maximum:
 $$ \xi_{\infty} (x) = \max \{\xi_{\infty} (z) : z \in \V' \}. $$
 Then, because~$\xi_{\infty} (x) \geq \xi_{\infty} (y)$ for all~$y \in N_{\infty} (x) \subset \V'$ and
 $$ \frac{1}{\card (N_{\infty} (x))} \sum_{y \in N_{\infty} (x)} (\xi_{\infty} (x)  - \xi_{\infty} (y)) = \xi_{\infty} (x) - \frac{1}{\card (N_{\infty} (x))} \sum_{y \in N_{\infty} (x)} \xi_{\infty} (y) = 0, $$
 we must have~$\xi_{\infty} (x) = \xi_{\infty} (y)$ for all~$y \in N_{\infty} (x)$.


\section{Optional stopping and consensus}
 As previously mentioned, the next step to prove the theorem is to apply the optional stopping theorem to the supermartingale~$X_t (\c)$.
 In order to apply the optional stopping theorem, we also define the collection of stopping times
 $$ T_{\ep} = \inf \,\{t : \norm{\xi_t (x) - \xi_t (y)} \not \in [\ep, \tau] \ \hbox{for all} \ (x, y) \in \E \} \quad \hbox{for all} \quad \ep \in (0, \tau). $$
 The next lemma shows that these stopping times are almost surely finite, which is one of the required assumptions to apply the optional stopping theorem.
\begin{lemma} --
\label{lem:stopping-time}
 For all~$\ep \in (0, \tau)$, time~$T_{\ep}$ is almost surely finite.
\end{lemma}
\begin{proof}
 On the event that~$T_{\ep} = \infty$, there exists~$(x, y) \in \E$ such that
\begin{equation}
\label{eq:stopping-time-1}
  \norm{\xi_t (x) - \xi_t (y)} \in [\ep, \tau] \quad \hbox{at arbitrary large times}.
\end{equation}
 This, together with Lemma~\ref{lem:limit}, implies that
\begin{equation}
\label{eq:stopping-time-2}
  \norm{\xi_{\infty} (x) - \xi_{\infty} (y)} = \lim_{t \to \infty} \norm{\xi_t (x) - \xi_t (y)} \in [\ep, \tau].
\end{equation}
 Observing also that, by the definition of~$\E_{\infty}$ in~\eqref{eq:subedge},
 $$ \begin{array}{rcl}
     (x, y) \in \E \ \hbox{and} \ \norm{\xi_{\infty} (x) - \xi_{\infty} (y)} \leq \tau & \Longrightarrow & (x, y) \in \E_{\infty} \vspace*{4pt} \\
                                  \norm{\xi_{\infty} (x) - \xi_{\infty} (y)} \geq \ep  & \Longrightarrow & \xi_{\infty} (x) \neq \xi_{\infty} (y), \end{array} $$
 we deduce that the limiting configuration~$\xi_{\infty}$ is not constant on the connected components of the random graph~$\G_{\infty}$ which, according to Lemma~\ref{lem:constant}, is an event with probability zero.
 In conclusion, the events in~\eqref{eq:stopping-time-1} and~\eqref{eq:stopping-time-2} and the event that~$T_{\ep} = \infty$ all have probability zero.
\end{proof} \\ \\
 Next, we prove that if, at time~$T_{\ep}$ for~$\ep$ small, the opinion of at least one vertex is close enough to the center~$\c$ of the opinion space, an event that we call~$\A$, then all the opinions across the network should be close to that opinion.
 This, together with a convexity argument and~Lemma~\ref{lem:constant}, implies consensus.
 In particular, the probability of consensus is larger than the probability of~$\A$, which will be used later to deduce the theorem.
\begin{lemma} --
\label{lem:consensus}
 For all~$\ep' \in (0, \tau / 2)$ and~$\tau > \r + \ep'$, there is~$\ep > 0$ such that
 $$ \A = \A (\ep, \ep') = \Big\{\norm{\xi_{T_{\ep}} (x) - \c} < \tau - \r - \ep' \ \hbox{for some} \ x \in \V \Big\} \subset \C. $$
\end{lemma}
\begin{proof}
 First, we observe that, on the event~$\A$, for all~$y \in \V$,
\begin{equation}
\label{eq:consensus-1}
  \norm{\xi_{T_{\ep}} (x) - \xi_{T_{\ep}} (y)} \leq \norm{\xi_{T_{\ep}} (x) - \c} + \norm{\xi_{T_{\ep}} (y) - \c} < \tau - \r - \ep' + \r = \tau - \ep'.
\end{equation}
 Now, we prove by induction that, on the event~$\A$,
\begin{equation}
\label{eq:consensus-2}
  \norm{\xi_{T_{\ep}} (x) - \xi_{T_{\ep}} (y)} \leq d (x, y) \,\ep \quad \hbox{for all} \quad y \in \V \quad \hbox{when} \quad \ep = \ep' / \card (\V).
\end{equation}
 Here,~$d (x, y)$ refers to the graph distance between vertex~$x$ and vertex~$y$. \vspace*{5pt} \\
{\bf Base case}: Assume that~$d (x, y) = 0$. Then,
 $$ y = x \quad \hbox{and} \quad \norm{\xi_{T_{\ep}} (x) - \xi_{T_{\ep}} (y)} = 0 = 0 \times \ep. $$
{\bf Induction step}: Let~$y \neq x$ and~$d = d (x, y)$ and assume that~\eqref{eq:consensus-2} holds for all vertices at graph distance~$d - 1$ from vertex~$x$.
 Letting~$z$ be a neighbor of~$y$ at distance~$d - 1$ from~$x$,
\begin{equation}
\label{eq:consensus-3}
  \norm{\xi_{T_{\ep}} (x) - \xi_{T_{\ep}} (z)} \leq (d - 1) \,\ep.
\end{equation}
 Because the longest self-avoiding path on the social network contains less than~$\card (\V)$ edges, we must have~$d < \card (\V)$.
 Assuming by contradiction that~$\norm{\xi_{T_{\ep}} (y) - \xi_{T_{\ep}} (z)} > \ep$, using the triangle inequality, recalling the definition of~$T_{\ep}$, and using~\eqref{eq:consensus-1}, we get
 $$ \begin{array}{rcl}
     \norm{\xi_{T_{\ep}} (x) - \xi_{T_{\ep}} (z)} & \n \geq \n & \norm{\xi_{T_{\ep}} (z) - \xi_{T_{\ep}} (y)} - \norm{\xi_{T_{\ep}} (y) - \xi_{T_{\ep}} (x)} \vspace*{4pt} \\
                                                  & \n \geq \n & \tau - (\tau - \ep') = \ep' = \ep \card (\V) > (d - 1) \,\ep,
    \end{array} $$
  which contradicts~\eqref{eq:consensus-3}, therefore~$\norm{\xi_{T_{\ep}} (y) - \xi_{T_{\ep}} (z)} \leq \ep$ and
  $$ \norm{\xi_{T_{\ep}} (x) - \xi_{T_{\ep}} (y)} \leq \norm{\xi_{T_{\ep}} (x) - \xi_{T_{\ep}} (z)} + \norm{\xi_{T_{\ep}} (y) - \xi_{T_{\ep}} (z)} \leq (d - 1) \,\ep + \ep = d \ep. $$
  This completes the proof of~\eqref{eq:consensus-2}. \vspace*{5pt} \\
  Recalling that~$d (x, y) < \card (\V)$ for all~$y \in \V$, we deduce that, on the event~$\A$,
  $$ \norm{\xi_{T_{\ep}} (x) - \xi_{T_{\ep}} (y)} < \ep \card (\V) = \ep' \quad \hbox{and so} \quad \xi_{T_{\ep}} (y) \in B (\xi_{T_{\ep}} (x), \ep') \quad \hbox{for all} \quad y \in \V. $$
  Since at each update of the process the new opinion is contained in the convex envelope of the other opinions, we deduce that, on the event~$\A$,
  $$ \xi_{\infty} (y) \in \conv \{\xi_{T_{\ep}} (z) : z \in \V \} \subset B (\xi_{T_{\ep}} (x), \ep') \quad \hbox{and so} \quad \norm{\xi_{\infty} (y) - \xi_{\infty} (z)} < 2 \ep' < \tau $$
  for all~$y, z \in \V$.
  In conclusion, $\G_{\infty}$ is connected and~$\C$ occurs by Lemma~\ref{lem:constant}.
\end{proof} \\ \\
 The last step to prove the theorem is to apply the optional stopping theorem to the supermartingale~$X_t (\c)$ stopped at time~$T_{\ep}$ to get a lower bound for the probability of the sub-event~$\A$.
\begin{lemma} --
\label{lem:subevent}
 The event~$\A = \A (\ep, \ep')$ defined in Lemma~\ref{lem:consensus} has probability
 $$ P (\A) \geq 1 - E \,\norm{X - \c} / (\tau - \r - \ep'). $$
\end{lemma}
\begin{proof}
 Observing that, on the complement of~$\A$,
 $$ \norm{\xi_{T_{\ep}} (x) - \c} \geq \tau - \r - \ep' \quad \hbox{for all} \quad x \in \V, $$
 we deduce that the conditional expectation of~$X_{T_{\ep}} (\c)$ is bounded from below by
 $$ E (X_{T_{\ep}} (\c) \,| \,\A^c) = E \bigg(\sum_{x \in \V} \,\norm{\xi_{T_{\ep}} (x) - \c \,} \,\Big| \,\A^c \bigg) \geq (\tau - \r - \ep') \cdot \card (\V). $$
 In particular, conditioning on the partition~$\{\A, \A^c \}$, we get
\begin{equation}
\label{eq:subevent-1}
  \begin{array}{rcl}
    E (X_{T_{\ep}} (\c)) & \n = \n & E (X_{T_{\ep}} (\c) \,| \,\A) P (\A) + E (X_{T_{\ep}} (\c) \,| \,\A^c) P (\A^c) \vspace*{4pt} \\
                         & \n \geq \n & E (X_{T_{\ep}} (\c) \,| \,\A^c) P (\A^c) \geq (\tau - \r - \ep') \card (\V) P (\A^c). \end{array}
\end{equation}
 Because the process~$X_t (\c)$ is a (bounded) supermartingale according to~Lemma~\ref{lem:supermartingale} and because~$T_{\ep}$ is an almost surely finite stopping time according to~Lemma~\ref{lem:stopping-time}, we may apply the optional stopping theorem.
 This, together with the previous inequality~\eqref{eq:subevent-1}, gives
\begin{equation}
\label{eq:subevent-2}
  P (\A) = 1 - P (\A^c) \geq 1 - \frac{E (X_{T_{\ep}} (\c))}{(\tau - \r - \ep') \card (\V)} \geq 1 - \frac{E (X_0 (\c))}{(\tau - \r - \ep') \card (\V)}.
\end{equation}
 Recalling that the opinions across the social network~$\G$ are initially independent and identically distributed with values in the opinion space~$\Delta$, we also have
\begin{equation}
\label{eq:subevent-3}
  E (X_0 (\c)) = E \bigg(\sum_{x \in \V} \,\norm{\xi_0 (x) - \c \,} \bigg) = \sum_{x \in \V} \,E \,\norm{\xi_0 (x) - \c \,} = E \,\norm{X - \c} \cdot \card (\V).
\end{equation}
 The result follows from~\eqref{eq:subevent-2} and~\eqref{eq:subevent-3}.
\end{proof} \\ \\
 To deduce the theorem, let~$\tau > \r$.
 Then, for all~$\ep' \in (0, \tau - \r)$, we have~$\tau > \r + \ep'$.
 Therefore, according to Lemmas~\ref{lem:consensus} and~\ref{lem:subevent}, there exists~$\ep > 0$ such that
 $$ P (\C) \geq P (\A (\ep, \ep')) \geq 1 - E \,\norm{X - \c} / (\tau - \r - \ep'). $$
 Because this holds for all~$\ep' > 0$ arbitrary small,
 $$ P (\C) \geq 1 - \lim_{\ep' \to 0} E \,\norm{X - \c} / (\tau - \r - \ep') = 1 - E \,\norm{X - \c} / (\tau - \r). $$
 This completes the proof of Theorem~\ref{th:consensus}.


\bibliographystyle{plain}
\bibliography{biblio.bib}

\end{document}